\documentclass[reqno,a4paper,draft]{amsart}
\usepackage{enumitem}

\setenumerate{label=\textnormal{(\arabic*)}}

\usepackage{amsmath,amssymb,dsfont,verbatim,mathtools,bm,geometry}

\usepackage[latin1]{inputenc}
\usepackage[raggedright]{titlesec}
\usepackage{mathtools}
\usepackage{tikz}
\usepackage{float,subfig}
\usepackage{amsrefs}

\titleformat{\chapter}[display]
{\normalfont\huge\bfseries}{\chaptertitlename\\thechapter}{20pt}{\Huge}
\titleformat{\section}
{\normalfont\Large\bfseries\center}{\thesection}{1em}{}
\titleformat{\subsection}
{\normalfont\large\bfseries}{\thesubsection}{1em}{}
\titleformat{\subsubsection}[runin]
{\normalfont\normalsize\bfseries}{\thesubsubsection}{1em}{}
\titleformat{\paragraph}[runin]
{\normalfont\normalsize\bfseries}{\theparagraph}{1em}{}
\titleformat{\subparagraph}[runin]
{\normalfont\normalsize\bfseries}{\thesubparagraph}{1em}{}
\titlespacing*{\chapter} {0pt}{50pt}{40pt}
\titlespacing*{\section} {0pt}{3.5ex plus 1ex minus .2ex}{2.3ex plus .2ex}
\titlespacing*{\subsection} {0pt}{3.25ex plus 1ex minus .2ex}{1.5ex plus .2ex}
\titlespacing*{\subsubsection}{0pt}{3.25ex plus 1ex minus .2ex}{1.5ex plus .2ex}
\titlespacing*{\paragraph} {0pt}{3.25ex plus 1ex minus .2ex}{1em}
\titlespacing*{\subparagraph} {\parindent}{3.25ex plus 1ex minus .2ex}{1em}


\newtheorem{theorem}{Theorem}[section]
\newtheorem{lemma}[theorem]{Lemma}
\newtheorem{proposition}[theorem]{Proposition}

\theoremstyle{definition}

\def\iproof{\begin{proof}}
\def\fproof{\end{proof}}

\def\td{\widetilde}

\def\la{\langle}
\def\ra{\rangle}
\def\Cx{\mathbb{C}}

\def\iarray{\begin{array}}
\def\farray{\end{array}}
\def\ieqna{\begin{eqnarray*}}
\def\feqna{\end{eqnarray*}}
\def\ieqn{\begin{eqnarray}}
\def\feqn{\end{eqnarray}}
\def\ienu{\begin{enumerate}}
\def\fenu{\end{enumerate}}
\def\iitem{\begin{itemize}}
\def\fitem{\end{itemize}}

\def\td{\widetilde}

\def\la{\langle}
\def\ra{\rangle}

\begin{document}

\title[Solutions of a polynomial system related to the Jacobian Conjecture]{The Groebner basis and solution set of a polynomial system related to the Jacobian Conjecture}

\author{Christian Valqui}
\address{Christian Valqui\\
Pontificia Universidad Cat\'olica del Per\'u, Secci\'on Matem\'aticas, PUCP,
Av. Universitaria 1801, San Miguel, Lima 32, Per\'u.}

\address{Instituto de Matem\'atica y Ciencias Afines (IMCA) Calle Los Bi\'ologos 245.
Urb San C\'esar. La Molina, Lima 12, Per\'u.}
\email{cvalqui@pucp.edu.pe}

\thanks{Christian Valqui was supported by PUCP-DGI-2023-PI0991}

\author{Valeria Ram\'irez}
\address{Valeria Ram\'irez\\
Pontificia Universidad Cat\'olica del Per\'u, Secci\'on Matem\'aticas, PUCP,
Av. Universitaria 1801, San Miguel, Lima 32, Per\'u.}
\email{valeria.ramireza@pucp.edu.pe}

\begin{abstract}
We compute the Groebner basis of a system of polynomial equations related to the Jacobian conjecture, and describe completely the solution set.
\end{abstract}

\maketitle

\section{Introduction}
Let $K$ be a characteristic zero field and let $K[y]((x^{-1}))$ be the algebra of Laurent series
in $x^{-1}$ with  coefficients in $K[y]$.
We will start from the following theorem, proved in~\cite{GGV}*{Theorem 1.9}.
\begin{theorem}\label{principal} The Jacobian conjecture in dimension two is false if and only if
there exist

\begin{itemize}

\smallskip

\item[-] $P,Q\in K[x,y]$ and $C,F\in K[y]((x^{-1}))$,

\smallskip

\item[-] $n,m\in \mathds{N}$ such that $n\nmid m$ and $m\nmid n$,

\smallskip

\item[-] $\nu_i\in K$ ($i=0,\dots,m+n-2$) with $\nu_0=1$,

\smallskip

\end{itemize}
such that
\begin{itemize}

\smallskip

\item[-] $C$ has the form
$$
C = x + C_{-1}x^{-1}+ C_{-2}x^{-2} + \cdots \qquad\text{with each $C_{-i}\in K[y]$,}
$$

\smallskip

\item[-] $gr(C)=1$ and $gr(F)=2-n$, where $gr$ is the total degree,

\smallskip

\item[-] $F_+=x^{1-n}y$, where $F_+$ is the term of maximal degree in $x$ of $F$,

\smallskip

\item[-] $C^n=P$ and $Q=\sum_{i=0}^{m+n-2}\nu_i C^{m-i}+F$.

\smallskip

\end{itemize}
Furthermore, under these conditions $(P,Q)$ is a counterexample to the Jacobian conjecture.
\end{theorem}
In~\cite{GGV}, the authors consider the following slightly more general situation. Let $D$ be a
$K$-algebra (for example, in Theorem~\ref{principal} we have $D=K[y]$),
$n,m$ positive integers such that $n\nmid m$ and $n\nmid m$, $(\nu_i)_{1\le i\le n+m-2}$ a family of elements
in $K$ with $\nu_0=1$ and $F_{1-n}\in D$ (in Theorem~\ref{principal} we take $F_{1-n}=y$).
A Laurent series in $x^{-1}$ of the form
$$
C = x + C_{-1}x^{-1}+ C_{-2}x^{-2} + \cdots \qquad\text{with $C_{-i}\in D$,}
$$
is \textsl{a solution of the system} $S(n,m,(\nu_i),F_{1-n})$, if there exist
 $P,Q\in D[x]$ and $F \in D[[x^{-1}]]$, such that
\begin{align*}
& F = F_{1-n} x^{1-n} + F_{-n} x^{-n} + F_{-1-n} x^{-1-n} +\cdots,\quad\text{with $F_{1-n},F_{-n},\dots$
in $D$}\\
&P=C^n\qquad\text{and}\qquad Q=\sum_{i=0}^{m+n-2}\nu_i C^{m-i}+F.
\end{align*}

For example, if $n=3$, then
\begin{align*}
P(\text{\bf{x}})=C^3=&\text{\bf{ x}}^3+3 C_{-1}\text{\bf{ x}}+ 3C_{-2}+(3C_{-1}^2+3C_{-3})\text{\bf{ x}}^{-1}+
(6C_{-1} C_{-2}+4C_{-4})\text{\bf{ x}}^{-2}\\
&+(C_{-1}^3+3C_{-2}^2+6c_{-2}C_{-3}+3C_{-5})\text{\bf{ x}}^{-3}\\
&+ (3C_{-1}^2 C_{-2}+6 C_{-2} C_{-3}+6 C_{-1} C_{-4}+6C_{-6})\text{\bf{ x}}^{-4}\\
&+ (3C_{-1} C_{-2}^2+3 C_{-1}^2 C_{-3}+3C_{-3}^2+6 C_{-2} C_{-4}+6 C_{-1} C_{-5}+6C_{-7})\text{\bf{ x}}^{-5}\\
& +\dots
\end{align*}
and the condition $C^3\in K[x]$ translates into the following conditions on $C_{-k}$:
\begin{align*}
0=(C^3)_{-1}=& \ 3C_{-1}^2+3C_{-3},\\
0=(C^3)_{-2}=& \ 6C_{-1} C_{-2}+4C_{-4},\\
0=(C^3)_{-3}=& \ C_{-1}^3+3C_{-2}^2+6c_{-2}C_{-3}+3C_{-5},\\
0=(C^3)_{-4}=& \ 3C_{-1}^2 C_{-2}+6 C_{-2} C_{-3}+6 C_{-1} C_{-4}+6C_{-6},\\
0=(C^3)_{-5}=& \ 3C_{-1} C_{-2}^2+3 C_{-1}^2 C_{-3}+3C_{-3}^2+6 C_{-2} C_{-4}+6 C_{-1} C_{-5}+6C_{-7},\\
\vdots
\end{align*}
In general, the condition $P(x)=C^n\in K[x]$ yields equations $(C^n)_{-k}=0$, whereas the condition
$Q(x)=\sum_{i=0}^{m+n-2}\nu_i C^{m-i}+F\in K[x]$ gives us the equations
$\left(\sum_{i=0}^{m+n-2}\nu_i C^{m-i}+F\right)_{-k}=0$, where we note that $F_{-k}=0$ for $k=1,\dots,n-2$.

It is easy to see (e.g.~\cite{GGV}*{Remark 1.13}) that the first $m+n-2$ coefficients determine
the others, i.e.,
the coefficients $C_{-1},\dots,C_{-m-n+2}$ determine univocally the coefficients $C_{-k}$ for $k>m+n-2$.
Moreover, the $F_{-k}$ for $k>n-1$ depend only on $F_{1-n}$ and $C$.
Consequently, having a solution $C$ to the system $S(n,m,(\nu_i),F_{1-n})$
is the same as having a solution  $(C_{-1},\dots,C_{-m-n+2})$ to the system
\begin{equation}
\begin{aligned}
&E_k:=(C^n)_{-k}  = 0, && \text{for $k=1,\dots, m-1$,}\\
&E_{m-1+k}:=\left(\sum_{i=0}^{m+n-2}\nu_i C^{m-i}\right)_{-k}  = 0, &&\text{for $k=1,\dots,n-2$,}\\
&E_{m+n-2}:=\left(\sum_{i=0}^{m+n-2}\nu_i C^{m-i}\right)_{1-n} + F_{1-n} = 0,
\end{aligned}\label{sistema de ecuaciones}
\end{equation}
with $m+n-2$ equations $E_k$ and $m+n-2$ unknowns $C_{-k}$.

In order to understand the solution set of this system, it would be very helpful to find a Groebner basis
for the ideal generated by the polynomials $E_k$ in $D[C_{-1},\dots,C_{m+n-2}]$.
In this paper we compute such a Groebner basis of~\eqref{sistema de ecuaciones} in a very particular case:
we assume
$n=3$, $m=3r+1$  or $m=3r+2$ for some integer $r>0$, and $\nu_i=0$
for $i>0$. Moreover we consider $D=\mathds{C}[y]$ and $F_{1-n}=y$, as in Theorem~\ref{principal}.
\newpage
\section{Computation of a Groebner basis for $I_{m-1}$}
\setcounter{equation}{0}
Assume
$n=3$, $3\nmid m>3 $ and $\nu_i=0$
for $i>0$. Set also $D=\mathds{C}[y]$ and $F_{1-n}=y$.

Then the system~\eqref{sistema de ecuaciones} reads
\begin{eqnarray}
\begin{array}{rcl}
E_i&=&\left\{\iarray{ll}
(C^3)_{-i},& i=1,\ldots,m-1\\
(C^{m})_{-1},& i=m,\\
(C^{m})_{-2}+y,& i=m+1,\\
\farray\right.\\
\end{array}\label{eqGGV}
\end{eqnarray}
where $(C^2)_{-i}$ denotes the coefficient of $x^{-i}$ in the Laurent series $C^3$. Explicitly, the
polynomials $E_i$ are given by
\begin{eqnarray}
\begin{array}{rl}
E_1&= \ 3C_{-1}^2+3C_{-3},\\
E_2&= \ 6C_{-1} C_{-2}+3C_{-4},\\
E_3&= \ C_{-1}^3+3C_{-2}^2+6C_{-1}C_{-3}+3C_{-5},\\
E_4&= \ 3C_{-1}^2 C_{-2}+6 C_{-2} C_{-3}+6 C_{-1} C_{-4}+3C_{-6},\\
E_5&= \ 3C_{-1} C_{-2}^2+3 C_{-1}^2 C_{-3}+3C_{-3}^2+6 C_{-2} C_{-4}+6 C_{-1} C_{-5}+3C_{-7},\\
& \vdots\\
E_{m-1}&=(C^3)_{1-m},\\
E_{m}&=(C^m)_{-1},\\
E_{m+1}&=(C^m)_{-2}+y,\\
\end{array}
\label{eqGGV1}
\end{eqnarray}

Each $E_i$ is a polynomial in the ring $\Cx[C_{-1},C_{-2},\ldots,C_{m+1},y]$, and the $m+1$ polynomials
yield the ideal
\ieqna
I&=&\la E_1,\ldots,E_{m},E_{m+1}\ra.
\feqna
Our goal  is to find a Groebner basis for the ideal $I$, but we find it nearly explicit only for
$I_{m-1}:=\la
E_1,E_2,\ldots,E_{m-2},E_{m-1}\ra$.  For this we note that the equations are homogeneous, for the weight obtained by setting
$$
w(C_{-i})=i+1,\quad \text{and}\quad w(y)=m+n-1=m+2.
$$
We consider $y$ as a variable, so the equations remain homogeneous.
Then
$$
w(E_k)=k+3, \text{ for $k=1,\dots, m-1$}, \quad w(E_{m})=m+1 \quad \text{and} \quad w(E_{m+1})=m+2.
$$

Note that for $k=1\dots, m-1$ we have
\begin{equation}\label{desarrollo de Ek}
        E_{k} := 3 \Big( \sum_{\substack{i=-1\\ 3i\ne k}}^{[[\frac{k+1}{2}]]} C_{-i}^{2} C_{-(k - 2i)} \Big) + 6 \Big(\sum_{\substack{ 0<i <j \\ i+j = k+1}} C_{-i} C_{-j} \Big) + 6 \Big(\sum_{\substack{ 0< i <j < l \\ i+j + l  = k}} C_{-i} C_{-j} C_{-l} \Big) + \varepsilon (C_{-\frac{k}{3}})^{3},
        \end{equation}
where    $\varepsilon = \begin{cases}
          1 , \quad 3 | k \\
          0,  \quad 3 \nmid k
        \end{cases}$.
Note that $C_1=1$ and $C_0=0$, and so
\begin{equation}\label{Ek desde Ck+2}
3\sum_{i=-1}^{[[\frac{k+1}{2}]]} C_{-i}^{2} C_{-(k - 2i)}=3 C_{k+2}+3\sum_{i=1}^{[[\frac{k+1}{2}]]} C_{-i}^{2} C_{-(k - 2i)}
\end{equation}
In order to compute a Groebner basis we will consider the degree reverse lexicographic monomial order, but for the degree given by the above mentioned weight. This means that the monomial order is given by the matrix
$$\text{wmat}=
\begin{pmatrix}
  m+2 & m+1 & m & \dots & 4 & 3 & 2 & m+2 \\
  0 & 0 & 0 & \dots & 0 & 0 & 0 & -1 \\
  0 & 0 & 0 & \dots & 0 & 0 & -1 & 0 \\
  0 & 0 & 0 & \dots & 0 & -1& 0  & 0 \\
  \vdots & \vdots & \vdots & \ddots & \vdots & \vdots & \vdots &\vdots \\
  0 & 0 & -1 & \dots & 0 & 0 & 0 & 0\\
  0 & -1 & 0 & \dots & 0 & 0 & 0 & 0
\end{pmatrix},
$$
on the variables $C_{-(m+1)},C_{-m},C_{-(m-1)},\dots, C_{-3},C_{-2}, C_{-1}, y$.
 We first compute the reduced Groebner basis
$(\td{E}_1,\td{E}_2,\ldots,\td{E}_{m-1})$ for the ideal $I_{m-1}:=\la
E_1,E_2,\ldots,E_{m-2},E_{m-1}\ra$.

\begin{proposition}
  The set $\{E_1,\dots,E_{m-1}\}$ is a Groebner basis of $I_{m-1}$.
  The reduced Groebner basis of $I_{m-1}$ is given by polynomials $\td{E}_k$ for
  $k=1,\dots,m-1$, each of the form
  $$
\td{E}_k=C_{-(k+2)}+R_k(C_{-1},C_{-2}),
$$
where $R_k(C_{-1},C_{-2})\in \Bbb{Q}[C_{-1},C_{-2}]$ is an homogeneous polynomial in the variables $C_{-1}$ and $C_{-2}$ of weight $w(\td{E}_k)=w(E_k)=k+3$.
\end{proposition}

\begin{proof}
By~\eqref{desarrollo de Ek} and~\eqref{Ek desde Ck+2} we know that $E_k$ is of the form
$$
E_k=3C_{-k-2}+T(C_{-1},\dots,C_{-k},\quad \text{for $k=1,\dots,m-1$},
$$
where $T$ is a polynomial in the variables $C_{-1},\dots,C_{-k}$.
Then by  Proposition 2.9.4 of~\cite{CLO}, since
$$
LCM(LT(E_i)/3,LT(E_j)/3)=LCM(C_{-i-2},C_{-j-2})= C_{-i-2}C_{-j-2}=(LT(E_i)/3)(LT(E_j)/3),
$$
we have $S(E_i,E_j)\longrightarrow_G 0$, and so, by Theorem 2.9.3 of~\cite{CLO}, the set $G=\{E_1/3,\dots,E_{m-1}/3\}$ is a Groebner basis of $I_{m-1}$.
One verifies directly that it is a minimal Groebner basis, according to Definition 2.7.4 of~\cite{CLO}. If we apply the process described in the proof of~\cite{CLO}*{Proposition 2.7.6} to the Groebner basis $G=\{E_1/3,\dots, E_{m-1}/3\}$ we obtain that
$$
\td{E}_1=\overline{E_1/3}^{G\setminus\{\frac {E_1}{3}\}}=E_1/3\quad\text{and}\quad \td{E}_1=\overline{E_2/3}^{G\setminus\{\frac {E_2}{3}\}}=E_2/3.
$$
Moreover, for $k=3,\dots,m-1$, we define $G_k=\{\td{E_1},\dots,\td{E_{k-1}},E_k,\dots,E_{m-1}\}$ and then
$$
\td{E}_k=\overline{E_k}^{G_{k}\setminus E_k},
$$
Clearly the remainder can have only the variables $C_{-1}$ and $C_{-2}$, hence $\td{E}_k$ is of the form
  $$
\td{E}_k=C_{-(k+2)}+R_k(C_{-1},C_{-2}),
$$
as desired.

\end{proof}

Although we have no explicit formula for $R_k(C_{-1},C_{-2})$, we can compute it for small $k$.

\begin{align*}
\tilde{E}_1&=C_{-3}+C_{-1}^2,\\
\tilde{E}_2&=C_{-4}+2C_{-1}C_{-2},\\
\tilde{E}_3&=C_{-5}+C_{-2}^2-\frac 53 C_{-1}^3,\\
\tilde{E}_4&=C_{-6}-5C_{-1}^2C_{-2},\\
\tilde{E}_5&=C_{-7}+\frac{10}{3}C_{-1}^4-5C_{-1}C_{-2}^2.\\
\end{align*}

Dividing the polynomials $E_m$ and $E_{m+1}$  by the polynomials $\{\td{E}_{m-1},\dots,\td{E}_2,\td{E}_{1}\}$ with respect to the given order, we obtain
$$
\overline{\frac{E_m}{3}}^{G_m\setminus\{\frac {E_m}{3}\}}=\td{E}_{m}=R_m(C_{-1},C_{-2})\quad \text{and} \quad
\overline{\frac{E_{m+1}}{3}}^{G_m\setminus\{\frac {E_{m+1}}{3}\}}=\td{E}_{m+1}=y+R_{m+1}(C_{-1},C_{-2}),
$$
where $R_m(C_{-1},C_{-2}),R_{m+1}(C_{-1},C_{-2})\in \Bbb{Q}[C_{-1},C_{-2}]$ are homogeneous polynomials such that
 $w(\td{E}_m)=w(E_m)=m+1$ and $w(\td{E}_{m+1})=w(E_{m+1})=m+2$.

Although we don't give an explicit description of the Groebner Basis of the whole system, in the next section we show how to determine the solution set of the polynomial system, using that
$$
I=\langle E_1,E_2,\dots, E_m, E_{m+1} \rangle= \langle \td{E}_1,\td{E}_2,\ldots,\td{E}_{m},\td{E}_{m+1} \rangle.
$$
\section{The solution set of the system of polynomial equations}
In this section we analyze the solutions of the system of equations. Note that the partial system $I_{m-1}$ shows that the values of
$C_{-1}$ and $C_{-2}$ determine univocally the values of $C_{-k}$ for $k>2$. Moreover, $C_{-1}$ and $C_{-2}$ can be computed using the following two equations:
\begin{equation} \label{primera ecuacion}
\td{E}_{m} = R_{m}(C_{-1},C_{-2})=0
\end{equation}
and
\begin{equation} \label{segunda ecuacion}
\td{E}_{m+1} = y + R_{m+1}(C_{-1},C_{-2})=0,
\end{equation}
where $R_{m}(C_{-1},C_{-2}), R_{m+1}(C_{-1},C_{-2}) \in \mathbb{Q}[C_{-1},C_{-2}]$ are homogeneous polynomials with respect to the weight considered
before, i.e. $w(C_{-1})=2$, $w(C_{-2})=3$. Moreover $w(\td{E}_m)= m+n-2= m+1$ and $w(\td{E}_{m+1})= m+n-1= m+2$. Then~\eqref{primera ecuacion} and~\eqref{segunda ecuacion} read
\begin{equation}\label{sumas homogeneas 1}
\td{E}_m=\sum_{2i + 3j = m+1} \lambda_{m}^{ij} C_{-1}^{i} C_{-2}^{j}
\end{equation}
and
\begin{equation}\label{sumas homogeneas 2}
\td{E}_{m+1}=y+\sum_{2i + 3j = m+2} \lambda_{m+1}^{ij} C_{-1}^{i} C_{-2}^{j},
\end{equation}
for some constants $\lambda_m^{ij}, \lambda_{m+1}^{ij}\in K$.
 By~\eqref{sumas homogeneas 2} the two variables cannot be zero at the same time. We compute first the solutions in the cases where one of the
variables is zero.\\

\noindent {\bf FIRST CASE}:  $C_{-1}=0$ and $C_{-2} \neq 0$. \\

In this case the only term surviving in~\eqref{sumas homogeneas 1}
is
$$
0=\td{E}_m=\lambda_m^{0j} C_{-2}^j,
$$
with $3j=m+1$. So necessarily
\begin{equation}\label{condition innecesaria}
\lambda_m^{0,(m+1)/3}=0\quad\text{if} \quad 3|m+1.
\end{equation}
Similarly, the only term surviving in the sum~\eqref{sumas homogeneas 2} has $i=0$,
and so we obtain
$$
0=\td{E}_{m+1}=y+ \lambda_{m+1}^{0j} C_{-2}^j\quad\text{with $3j=m+2$.}
$$
 Since $y\ne 0$, necessarily $ \lambda_{m+1}^{0j}\ne 0$ for $3j=m+2$, and so $3|m+2$, i.e. $m\equiv 1\mod 3$. This shows that
the condition~\eqref{condition innecesaria} is trivially satisfied.
\begin{lemma}
If $3|m+2$, and $C_{-1}=0$, then $\lambda_{m+1}^{0j}\ne 0$ for $3j=m+2$.
\end{lemma}
\begin{proof}
  It is easy to check that $P=x^3+3C_{-2}$, and then, by Newtons binomial theorem we have
\begin{equation}\label{Newton 1}
  C^m=P^{m/3}=\sum_{k=0}^\infty \binom{m/3}{k}(3C_{-2})^k(x^3)^{\frac m3-k}.
\end{equation}
Thus $\lambda_{m+1}^{0j}C_{-2}^j=(C^m)_{-2}$ is the coefficient of $x^{-2}=(x^3)^{\frac m3-j}$, since $m=3j-2$.
Then
$$
\lambda_{m+1}^{0j}=\binom{m/3}{j}3^j\ne 0,
$$
as desired.
\end{proof}

Thus we have proved the following proposition.
\begin{proposition}
If $(C_{-1},C_{-2},\dots,C_{-(m+1)})$ is a solution of the system~\eqref{eqGGV}, with  $C_{-1}=0$ and $C_{-2}\ne 0$, then
\begin{itemize}
\item $m\equiv 1\mod 3$,
\item $\lambda_{m+1}^{0j}\ne 0$ for $j:=\frac{m+2}{3}$,
\item There are $j$ solutions of the system~\eqref{eqGGV} in $K[y^{1/j}]$, given by
$$
C_{-1}=0,\quad C_{-2} = \left(\frac{-y}{\lambda^{0j}_{m+1}} \right)^{\frac{1}{j}}\quad\text{and}\quad
C_{-k} = - R_{k-2}(C_{-1},C_{-2}) \quad \text{for}\quad 3\le k \leq m+1.
$$
\end{itemize}
\end{proposition}

\noindent {\bf SECOND CASE}:  $C_{-1}\ne 0$ and $C_{-2} = 0$. \\

\noindent In this case the only term surviving in~\eqref{sumas homogeneas 1}
is
$$
0=\td{E}_m=\lambda_m^{i0} C_{-1}^i,
$$
with $2i=m+1$. So necessarily
\begin{equation}\label{condition innecesaria}
\lambda_m^{(m+1)/2,0}=0\quad\text{if} \quad 2|m+1.
\end{equation}
Similarly, the only term surviving in the sum~\eqref{sumas homogeneas 2} has $j=0$,
and so we obtain
$$
0=\td{E}_{m+1}=y+ \lambda_{m+1}^{i0} C_{-1}^i\quad\text{with $2i=m+2$.}
$$
 Since $y\ne 0$, necessarily $ \lambda_{m+1}^{i0}\ne 0$ for $2i=m+2$, and so $2|m+2$, i.e. $m$ is even. This shows that
the condition~\eqref{condition innecesaria} is trivially satisfied.
\begin{lemma}
If $2|m$ and $C_{-2}=0$, then $\lambda_{m+1}^{i0}\ne 0$ for $2i=m+2$.
\end{lemma}
\begin{proof}
  It is easy to check that $P=x^3+3xC_{-1}$, and then, by Newtons binomial theorem we have
\begin{equation}\label{Newton 1}
  C^m=P^{m/3}=\sum_{k=0}^\infty \binom{m/3}{k}(3xC_{-1})^k(x^3)^{\frac m3-k}.
\end{equation}
Thus $\lambda_{m+1}^{i0}C_{-1}^i=(C^m)_{-2}$ is the coefficient of $x^{-2}=(x)^i(x^3)^{\frac m3-i}$, since $m=2i-2$.
Then
$$
\lambda_{m+1}^{i0}=\binom{m/3}{i}3^i\ne 0,
$$
as desired.
\end{proof}

Thus we have proved the following proposition.
\begin{proposition}
If $(C_{-1},C_{-2},\dots,C_{-(m+1)})$ is a solution of the system~\eqref{eqGGV}, with  $C_{-1}\ne 0$ and $C_{-2}= 0$, then
\begin{itemize}
\item $m\equiv 1\mod 3$,
\item $\lambda_{m+1}^{i0}\ne 0$ for $i:=\frac{m+2}{2}$,
\item There are $i$ solutions of the system~\eqref{eqGGV} in $K[y^{1/i}]$, given by
$$
C_{-1} = \left(\frac{-y}{\lambda^{i0}_{m+1}} \right)^{\frac{1}{i}},\quad  C_{-2}=0 \quad\text{and}\quad
C_{-k} = - R_{k-2}(C_{-1},C_{-2}) \quad \text{for}\quad 3\le k \leq m+1.
$$
\end{itemize}
\end{proposition}

\noindent {\bf THIRD CASE}:  $C_{-1}\ne 0$, $C_{-2} \ne 0$ and $m$ even. \\
In this case we introduce a new auxiliary variable $t$ satisfying $C_{-2}^{2} = t C_{-1}^{3}$.
The equality~\eqref{sumas homogeneas 1} now reads
$$
\td{E}_{m} = \sum_{2i + 3j = m+1} \lambda_{m}^{ij} C_{-1}^{i} C_{-2}^{j}=
\sum_{2i + 6r+3 = m+1} \lambda_{m}^{i,2r+1} C_{-1}^{i} C_{-2}^{2r+1}=
\sum_{2i + 6r+2 = m} \lambda_{m}^{i,2r+1} C_{-1}^{i+3r} C_{-2} t^r,
$$
since $m$ even implies that the weight $2i+3j=m+1$ is odd, so $j$ is odd and can be written as $2r+1$.  Moreover, for the terms in the sum
we have
$i+3r = \frac{m-2}{2}$, and so we arrive at
$$
0 = C_{-1}^{\frac{m-2}{2}} C_{-2} \sum_{\substack{2i + 6r = m-2 \\ j = 2r +1}} \lambda^{ij}_{m} t^{r}.
$$
Thus $t$ is a root of the polynomial
\begin{equation}\label{definicion de f}
f(t)=\sum_{r=0}^{\lfloor \frac{m-2}{6} \rfloor} a_r t^r,\quad \text{where $a_r=\lambda_m^{\frac{m-2-6r}{2}, 2r+1}$.}
\end{equation}
 Let $\{t_1,\dots,t_s\}$ be the roots of the polynomial $f(t)$.
Note that in the equality~\eqref{sumas homogeneas 2} the power $j$ has to be even, since $m$ is even and $2i+3j=m+2$. Hence,
if we replace $C_{-2}^2$ by $t_l C_{-1}^3$ in~\eqref{sumas homogeneas 2}, we obtain
		   \begin{equation*}
		   	\td{E}_{m+1} = y+ \sum_{\substack{2i + 3j =m+2 \\ j = 2r}} \lambda_{m+1}^{ij} C_{-1}^{i}C_{-2}^{j} = y+ \sum_{2i+6r=m+2} \lambda_{m+1}^{i,2r} C_{-1}^{i+3r} t_{l}^{r}.
		   \end{equation*}
Note that for each of the terms in the last sum we have $i+3r=\frac{m+2}{2}$, and so
\begin{equation*}
		   	0 = y + C_{-1}^{\frac{m+2}{2}} g(t_l),\quad\text{where $g(t)=\sum_{r=0}^{\lfloor \frac{m+2}{6} \rfloor} b_r t^r$,}
		   \end{equation*}
 with $b_r=\lambda_{m+1}^{\frac{m+2-6r}{2},2r}$. It follows that		
		   \begin{equation*}
		   	C_{-1} = \Big(\frac{-y}{g(t_l)} \Big)^{\frac{2}{m+2}}.
		   \end{equation*}
		Thus we have arrived at the following result.
\begin{proposition}
If $(C_{-1},C_{-2},\dots,C_{-(m+1)})$ is a solution of the system~\eqref{eqGGV}, with  $C_{-1}\ne 0$, $C_{-2}\ne 0$ and $m$ even, then
the system has at most $s\cdot (m+2)$ solutions, where $s$ is the number of roots of $f(t)$ defined in~\eqref{definicion de f}.
Moreover, for every choice of a root $t_l$ of $f$, the solutions are given by
\begin{align*}
 C_{-1} &= \Big(\frac{-y}{g(t_l)} \Big)^{\frac{2}{m+2}} ,& \frac{m+2}{2}\quad \text{choices,}  \\
  C_{-2} &= \left( t_l C_{-1}^3\right)^{\frac 12} ,& 2\quad \text{choices,}\\
C_{-k} &= - R_{k-2}(C_{-1},C_{-2}) \quad \text{for}\quad 3\le k \leq m+1.
\end{align*}
\end{proposition}

\noindent {\bf FOURTH CASE}:  $C_{-1}\ne 0$, $C_{-2} \ne 0$ and $m$ odd. \\
In this case we introduce a new auxiliary variable $t$ satisfying $C_{-2}^{2} = t C_{-1}^{3}$.
The equality~\eqref{sumas homogeneas 1} now reads
$$
\tilde{E}_{m} = \sum_{2i + 3j = m+1} \lambda_{m}^{ij} C_{-1}^{i} C_{-2}^{j}=
\sum_{2i + 6r = m+1} \lambda_{m}^{i,2r} C_{-1}^{i} C_{-2}^{2r}=
\sum_{2i + 6r = m+1} \lambda_{m}^{i,2r} C_{-1}^{i+3r} t^r,
$$
since $m$ odd implies that the weight $2i+3j=m+1$ is even, so $j$ is even and can be written as $2r$.  Moreover, for the terms in the sum
we have
$i+3r = \frac{m+1}{2}$, and so we arrive at
$$
0 = C_{-1}^{\frac{m+1}{2}} \sum_{\substack{2i + 6r = m+1 \\ j = 2r }} \lambda^{ij}_{m} t^{r}.
$$
Thus $t$ is a root of the polynomial
\begin{equation}\label{definicion de f 2}
f(t)=\sum_{r=0}^{\lfloor \frac{m+1}{6} \rfloor} a_r t^r,\quad \text{where $a_r=\lambda_m^{\frac{m+1-6r}{2}, 2r}$.}
\end{equation}
 Let $\{t_1,\dots,t_s\}$ be the roots of the polynomial $f(t)$.
Note that in the equality~\eqref{sumas homogeneas 2} the power $j$ has to be odd, since $m$ is odd and $2i+3j=m+2$. Hence,
if we replace $C_{-2}^2$ by $t_l C_{-1}^3$ in~\eqref{sumas homogeneas 2}, we obtain
		   \begin{equation*}
		   	\td{E}_{m+1} = y+ \sum_{\substack{2i + 3j =m+2 \\ j = 2r+1}} \lambda_{m+1}^{ij} C_{-1}^{i}C_{-2}^{j} = y+ \sum_{2i+6r+3=m+2} \lambda_{m+1}^{i,2r+1} C_{-1}^{i+3r} C_{-2} t_{l}^{r}.
		   \end{equation*}
Note that for each of the terms in the last sum we have $i+3r=\frac{m-1}{2}$, and so
\begin{equation*}
		   	0 = y + C_{-1}^{\frac{m-1}{2}} C_{-2} g(t_l),\quad\text{where $g(t)=\sum_{r=0}^{\lfloor \frac{m-1}{6} \rfloor} b_r t^r$,}
		   \end{equation*}
 with $b_r=\lambda_{m+1}^{\frac{m-1-6r}{2},2r+1}$. We also replace $C_{-2}$ by $\left(t_l C_{-1}^3\right)^{\frac 12}$. It follows that		
		   \begin{equation*}
0= y+ C_{-1}^{\frac{m+2}{2}} (t_{l})^{\frac 12} g(t_l),
		   \end{equation*}
and so
$$
C_{-1}=\Big(\frac{-y}{(t_{l})^{\frac 12} g(t_l)}  \Big)^{\frac{2}{m+2}}.
$$
Thus we have arrived at the following result.
\begin{proposition}
If $(C_{-1},C_{-2},\dots,C_{-(m+1)})$ is a solution of the system~\eqref{eqGGV}, with  $C_{-1}\ne 0$, $C_{-2}\ne 0$ and $m$ odd, then
the system has at most $2\cdot s\cdot (m+2)$ solutions, where $s$ is the number of roots of $f(t)$ defined in~\eqref{definicion de f 2}.
Moreover, for every choice of a root $t_l$ of $f$, we first choose a square root of $t_l$ and then the solutions are given by
\begin{align*}
 C_{-1} &= \Big(\frac{-y}{(t_{l})^{\frac 12} g(t_l)}  \Big)^{\frac{2}{m+2}} ,& \frac{m+2}{2}\quad \text{choices,}  \\
  C_{-2} &= \left( t_l C_{-1}^3\right)^{\frac 12} ,& 2\quad \text{choices,}\\
C_{-k} &= - R_{k-2}(C_{-1},C_{-2}) \quad \text{for}\quad 3\le k \leq m+1.
\end{align*}
\end{proposition}

\end{document}